\documentclass[10pt, reqno]{amsart}
\usepackage[a4paper,inner=3cm,outer=3cm,top=3cm,bottom=3cm]{geometry}
\pdfoutput=1
 \usepackage{amsaddr}
\usepackage{a4wide}
\usepackage{hyperref}
\usepackage{fancyvrb}
\usepackage{graphicx}				
\usepackage{amssymb}

\usepackage{amsthm}
\usepackage{amsmath}
\usepackage{url}
\usepackage{enumerate}
\usepackage{mathtools}
\usepackage{algorithm}
\usepackage{color}
\usepackage{caption}

\newtheorem{theorem}{Theorem}[section]

\theoremstyle{definition}
\newtheorem{definition}{Definition}[section]
\newtheorem{lemma}{Lemma}[section]
\newtheorem{corollary}{Corollary}[section]
\newtheorem{remark}{Remark}[section]

\newcommand{\E}{\mathbf E}

\renewcommand{\P}{\mathbf P}

\newcommand{\R}{\mathbb R}

\newcommand{\e}{\mathrm e}

\newcommand{\N}{\mathbb N} 

\mathcode`e=`e
\newcommand*\diff{\mathop{}\!\mathrm{d}}

\begin{document}

\title[Metric dimension of some large trees]{Metric dimension of critical Galton-Watson trees and linear preferential attachment trees}
\author{J\'{u}lia Komj\'{a}thy$^\star$ and Gergely \'Odor$^\dagger$}
 
\address{$^\star$ Department of Mathematics and Computer Science, Eindhoven University of Technology, The Netherlands\\
$^\dagger$ School of Computer and Communication Sciences, \'Ecole Polytechnique F\'ed\'erale de Lausanne (EPFL)
}
\thanks{$\star$ Postal address: PO Box 513
5600 MB Eindhoven, The Netherlands}
\thanks{$\dagger$ Postal address: EPFL IC IINFCOM INDY2
Station 14
CH-1015 Lausanne, Switzerland}
\thanks{Acknowledgements. The work of G. \'Odor was supported by the Swiss National Science Foundation under grant number 200021-182407.}
\email{j.komjathy@tue.nl, gergely.odor@epfl.ch}


\maketitle

\begin{abstract}
The metric dimension of a graph $G$ is the minimal size of a subset $R$ of vertices of $G$ that, upon reporting their graph distance from a distingished (source) vertex $v^\star$, enable unique identification of the source vertex $v^\star$ among all possible vertices of $G$. In this paper we show a Law of Large Numbers (LLN) for the metric dimension of some classes of trees: critical Galton-Watson trees conditioned to have size $n$, and growing general linear preferential attachment trees. The former class includes uniform random trees, the latter class includes Yule-trees (also called random recursive trees), $m$-ary increasing trees, binary search trees, and positive linear preferential attachment trees. In all these cases, we are able to identify the limiting constant in the LLN explicitly. Our result relies on the insight that the metric dimension can be related to subtree properties, and hence we can make use of the powerful fringe-tree literature developed by Aldous and Janson \emph{et al}. 
\end{abstract}      
\section{Introduction}
The metric dimension is a notion originating from combinatorics, first defined by Slater \cite{slater1975leaves} and independently by Harary and Melter \cite{harary1976metric}. 
Heuristically, the problem can be described in terms of \emph{source-detection}: given a graph $G=(V,E)$ with an unknown special vertex $v^\star$, we would like to identify $v^\star$ based on limited observations. We think of $v^\star$ as the source of a spreading process (say, a color red, that can be thought of being an infection, or any type of information) on the graph. The spreading starts at time $t=t_0=0$, when $v^\star$ becomes red. The color then spreads at unit speed across edges: each  direct neighbor of $v^\star$ is colored red at time $t=1$, each second neighbor at time $t=2$, and so on. Vertices keep their color forever. We are allowed to place, in advance,  \emph{sensor vertices} on the graph, forming a sensor set $R\subset V$. Sensor vertices report  their coloring/infection time. Based on the vector of these infection times, we would like to uniquely identify the source vertex $v^\star$. The minimal number of sensors needed for perfect detection, no matter what the location of $v^\star$, is called the \emph{metric dimension} (MD) of the graph, that we denote by $\beta(G)$. Any set of sensors that can uniquely identify the source vertex $v^\star$ (no matter what its location is) is called a \emph{resolving set}.  

\emph{Algorithmic aspects.} Computing resolving sets or even the metric dimension for general graphs is shown to be NP-hard \cite{khuller1996landmarks} and it is approximable only to a factor of $\log(N)$ \cite{beerliova2006network,hauptmann2012approximation}. The MD of specific deterministic graph families has been extensively studied, we refer to \cite{8391927} for a list of references. For instance, for trees the MD can be written as the difference of the number of leaves and so-called exterior major vertices of the tree (vertices of at least degree $3$ that have a line-graph leading to a leaf), both of which can be computed in linear time \cite{khuller1996landmarks}.  We mention that the MD has deep connections to the automorphism group of the graph $G$ \cite{bailey2011base,caceres2010determining,garijo2014difference}, and hence the graph isomorphism problem \cite{babai1980random}.
 
\emph{Asymptotic results.} From the probabilistic point of view little is known about the \emph{asymptotic behaviour} of MD  of random graph families as their sizes tend to infinity.
 A pioneering work \cite{bollobas2012metric} determines the asymptotics of MD of Erd\H{o}s-R\'enyi random graphs. In this Law of Large Numbers (LLN) type of result, the authors showed a surprising non-monotonous zig-zag phenomenon of the metric dimension as the average degree increases from bounded to linear in the graph size. 
A central limit theorem (CLT) type result for uniform random trees was determined in \cite{mitsche2015limiting}, and also for  subcritical Erd\H{o}s-R\'enyi random graphs.

\emph{Our contribution.} In this paper, we provide LLN type results for two general distributions on trees: families of growing trees that grow according to general linear preferential attachment schemes, and conditioned critical Galton-Watson trees that include uniform random trees.

We describe these families briefly. 
 In a general linear preferential attachment tree, there are two parameters, $\rho>0$ and $\chi\in \R$.  We start with a single root vertex. When there are $i$ vertices, we attach the $(i+1)$-st vertex to one of the existing vertices $v\le i$ with probability proportional to $(\rho+\chi\mathrm{deg}_i(v))$, where $\mathrm{deg}_i(v)$ is degree of vertex $v$ after $i$ vertices have been added. 
 Clearly, due to the normalization, only the quotient $\rho/\chi$ matters, and for the rest of the paper wlog we only consider $\chi \in \{-1,0,1\}$. When $\chi=-1$, we require $\rho$ to be an integer. 

We explain now why this class of trees contain $m$-ary increasing trees, binary search trees, and uniform recursive trees as well as rich-get-richer trees, that are the `usual'   linear preferential attachment trees. When we take $\rho=m$ and $\chi=-1$, we obtain the  \emph{$m$-ary increasing tree}: In its original definition of an  \emph{$m$-ary increasing tree} each vertex has the potential to have $m$ labeled offspring. The tree starts with a single vertex (the root) at step $1$, and at each step a new vertex arrives. When the tree has $i$ vertices, a new vertex can attach to $mi-(i-1)=(m-1)i+1$ possible places, since out of the $mi$ possible places, $i-1$ are already taken (only the root does not have a parent). An \textit{ $m$-ary increasing tree} with $n$ vertices, is constructed by starting with a single root vertex, and placing the $(i+1)$-st vertex uniformly randomly among the $(m-1)i+1$ possible places \cite{holmgren2017fringe}. The probability that the $(i+1)$-st vertex connects to vertex $v\le i$ is thus proportional to $m-\deg_i(v)$. Hence, we recognise the formula for $m=\rho$ and $\chi=-1$.

For $\rho=2$ and $\chi=-1$, the binary increasing tree corresponds to another well-known tree:  the random \emph{binary search tree}, an object that gained attention in computer science. In (the original definition of) a binary search tree, each vertex can store a single key and can have at most two children. The keys can be thought of as i.i.d.\ uniform random  variables on $[0,1]$ (this is a representation used by Devroye in \cite{Devr02}). Initially, the first key $K_1$ arrives and is placed at the root. This makes the root a \emph{full} vertex. Upon filling, every vertex creates two \emph{potential} vertices, one on the left and one on the right, that can receive a key each. These potential vertices do not count as part of the tree yet, only once they contain a key and become full vertices.  After the tree has $i$  keys, the $(i+1)$-st key $K_{i+1}$ arrives and is compared to the key in the root. If $K_{i+1}<K_1$, it is pushed to the left (otherwise to the right). Then it is compared to the key occupying the vertex that is the left (resp. right) child of the root, and again pushed left (resp. right) if it is less (resp. larger) than the key in that vertex. The procedure continues until the key finds an potential vertex and occupies it. Since only the permutation of the keys matters, it can be shown that when the tree has $i$ full vertices, and hence $i+1$ potential vertices, the $(i+1)$-st vertex is equally likely to be placed at any of these potential vertices. Hence, the probability that a full vertex with  $v\le i$ with degree\footnote{Potential vertices do not contribute to the degree, only full vertices do.} $\deg_i(v)$ gets a new child in step $i+1$ has probability $(2-\deg_i(v))/(i+1)$,  and we get back $\rho=2$, $\chi=-1$.  

A similar construction exists for $m>2$, called the $m$-ary search tree, when each vertex can store up to $m-1$ keys. This tree, however,  is not equivalent to the $m$-ary increasing tree \cite{holmgren2017fringe}, and we omit studying them further in this paper. Binary search trees are also the tree-representation of the Quicksort algorithm \cite{Knut98}.  Many of their properties are well studied, including Law of Large Numbers and Central Limit Theorems, see e.g. \cite{Devr02, Fuch08} such as the proportion of $k$-protected nodes or subtree sizes.

 The \emph{random recursive tree} is constructed analogously to the previous construction, except there is no dependence on the degree: starting with a single root vertex, the $(i+1)$-st vertex attaches uniformly to each of the already present $i$ vertices by an edge. This case corresponds to $\rho=1, \chi=0$.
Random recursive trees have a natural correspondence to binary search trees, and so often they are treated together \cite{Knut97}. They are also called Yule-trees, due to the fact that they can be naturally embedded in a Yule-process, and hence they have connections to phylogenetic trees \cite{BluFraJan06}.  

 The `usual' linear preferential attachment tree, also called \emph{rich-get-richer} tree, is constructed by taking $\rho>0$, $\chi=1$. In this case the $(i+1)$-st vertex attaches to $v\le i$ with probability proportional to $\rho+\deg_i(v)$.
 The $\rho=\chi=1$ case corresponds to the positive linear preferential attachment tree, which was informally introduced by Barab\'asi and Albert \cite{barabasi1999emergence}, although they allowed general graphs, not only trees. This is the model that produces power-law degree distributions \cite{BolRioSpe01}, see also Hofstad \cite{Hofs17} and the survey \cite{holmgren2017fringe}.
 Positive linear preferential attachment trees have already been studied in the context of source location \cite{jog2016analysis}, with the difference that the authors of \cite{jog2016analysis} consider snapshot-based source location and the MD is connected to sensor-based source location \cite{zejnilovic2013network}. 
 
The survey \cite{holmgren2017fringe} gives and excellent overview of the literature on various properties of all these growing trees, hence we refer the reader there for further literature.

Our main results can be summarised in the following two meta-theorems. 
\begin{theorem}[Meta-theorem about growing trees]
\label{thm:meta} Let $(\mathcal{T}_n^{(\rho,\chi)})_{n\ge 1}$ be a sequence of random growing general linear preferential attachment trees with $n$ vertices, with growth parameters $\rho>0$ and $\chi\in\{-1,0,1\}$, with $\rho\in \N$ when $\chi=-1$. Then,
\begin{equation}
\label{eq:meta}
\frac{\beta(\mathcal{T}^{(\rho,\chi)}_n)}{n} \xrightarrow{a.s. } c_{(\rho,\chi)}
\end{equation}
where $c_{(\rho,\chi)}\in(0,\infty)$ is a  constant that we determine \emph{explicitly}.
\end{theorem} 
We mention that our method provides almost sure LLN for a much larger class of random growing trees.  This class is the class of trees that can be embedded in a Crump-Mode-Jagers branching process with finite Malthusian parameter; e.g. sub-linear preferential attachment trees, $m$-ary search trees, fragmentation trees, etc. We refer the reader to various classes of such trees to the survey of Janson and Holmgren \cite{holmgren2017fringe}.

Our second result is motivated by reproducing LLN of the metric dimension of \emph{uniform random trees}  \cite{mitsche2015limiting}. A uniform random tree on $n$ vertices is a tree that is chosen uniformly at random (u.a.r.) from the possible $n^{n-2}$ labeled trees on $n$ vertices. As mentioned before, LLN and even CLT for the MD of uniform random tree was proved in \cite{mitsche2015limiting} using analytic combinatorics. We are able to reproduce the LLN result with a very short proof, and in higher generality. 

Namely a uniform random tree has the same distribution as a Galton-Watson branching process, with Poisson offspring distribution with mean $1$, conditioned to have total progeny $n$, see e.g. \cite[Proof of Theorem 3.17]{Hofs17}. Hence it is equivalent to determine the MD of conditioned GW trees.

 A \textit{Galton-Watson tree} is a random tree defined by the offspring distribution $\xi$ taking values in $\mathbb N=\{0,1,\dots\}$. Initially a single individual (vertex) is born, which becomes the root of the tree, and the root gives rise to $\xi$ children. Thereafter, each newly born individual samples its own independent copy of $\xi$ and gives rise to that many new children, and the process continues recursively. We consider Galton-Watson trees conditioned to have $n$ vertices, so we must assume that $\P(\xi=0)\ne 0$, otherwise the process never ends. We will assume that the Galton-Watson trees are critical, i.e., $\E[\xi]=1$, which is also fairly natural for conditioned Galton-Watson trees (see Remark 3.1 of \cite{janson2013asymptotic}), since in this case a non-trivial limiting measure on trees exists (called the incipient infinite tree).

\begin{theorem}[Conditioned Galton-Watson trees]
\label{thm:GW}
Let $\mathcal{GW}_n$ be a sequence of critical Galton-Watson trees conditioned to have $n$ vertices, with offspring distribution $\xi$, where $\E[\xi]=1$ and $\E[\xi^2]<\infty$. Let $p_k=\P(\xi=k)$ for $k \in \mathbb{N}$. Then, 
\begin{equation}
\frac{\beta(\mathcal{GW}_n)}{n} \xrightarrow{p} p_0-1+G_\xi\left(1-\frac{p_0}{1-p_1}\right)+\frac{p_1p_0}{1-p_1}.
\end{equation}
where $G_\xi(x)=\sum_{n=0}^\infty p_n x^n$ is the probability generating function of $\xi$ evaluated at $x$.
\end{theorem}

As a corollary of this theorem, by substituting $\xi=\mathrm{Poi}(1)$ we recover the result of \cite{mitsche2015limiting} on uniform random trees.

\begin{corollary}
The metric dimension  of a uniform random tree $\mathcal U_n$ on $n$ vertices satisfies the following Law of Large Numbers:  \[ \frac{\beta(\mathcal U_n)}{n} \xrightarrow{p} e^{-1}-1+e^{-\frac{1}{e-1}}+ \frac{e^{-1}}{e-1} \approx 0.14076941.\]
\end{corollary}

\emph{Methodology.} 
The metric dimension of a given fixed tree can be computed explicitly using the number of leaves of the tree and the number of exterior major vertices, i.e., vertices of at least degree $3$ that have a line-graph leading to a leaf, see Theorem \ref{thm_slater} below.

The novel insight in our proofs is that both the asymptotic proportion of leaves as well as that of exterior major vertices of random trees $\mathcal{T}$ can be computed using results from the \emph{fringe tree literature} initiated by Aldous in \cite{aldous1991asymptotic}. 
A fringe tree of a rooted tree, in plain words, is the random subtree obtained by choosing a vertex u.a.r.\ in the tree and taking its subtree pointing away from the root. The distribution of fringe trees is shown to converge for a large class of trees. So, fringe trees of a \emph{rooted tree} $\mathcal{T}$ helps us to compute the asymptotic proportion of vertices $v$ in $\mathcal{T}$ that have a certain property $\mathcal{P}$, with the limitation that $\mathcal{P}$ must be a \emph{subtree-property}. A subtree property is any property that depends only on the subtree of $\mathcal{T}$ rooted at $v$ pointing away from the root. It is easy to see that being a leaf is a subtree-property. While strictly speaking being an exterior major vertex is not a subtree-property, we find a subtree property that serves as a good proxy.

The use of fringe tree-methodology allows us to use probabilistic arguments that are often much shorter than the analytic-combinatorial arguments used in \cite{mitsche2015limiting}: the proportion of fringe-trees satisfying a given subtree property converges. Moreover, since the fringe distribution of several general random tree families are known \cite{holmgren2017fringe, janson2013asymptotic}, our proofs are quite general. Our results hold for critical Galton-Watson trees with a finite variance degree distribution (which includes, among others, uniform random trees, Motzkin trees, random binary trees) and all linear preferential attachment trees (which includes, among others, binary search trees, random recursive trees, positive linear preferential attachment trees) \cite{drmota2009random}.

The fringe tree literature has CLT type results, which suggests that many of our results in this paper can also be extended to a CLT. In particular, the CLT of metric dimension for binary search trees and uniform recursive trees should be a consequence of the CLT proved in \cite{holmgren2015limit}. For the other cases, this is not a trivial extension, and we leave it for future work.

\emph{Other contexts.} Resolving sets have wide rage of applications, including robot navigation \cite{khuller1996landmarks, shao2019metric}, computational chemistry \cite{chartrand2000resolvability}, network discovery \cite{beerliova2006network} and source detection. In particular, source detection has a large body of literature. From the statistical point of view, motivated by the problem of determining the authors of online viruses, malicious information, and fake news, the seminal work  \cite{ShaZam09} investigated the question: \emph{Can we locate the source if we only observe the epidemic much later, when it has already infected a large fraction of the population?} 
Various statistical estimators of the source have been developed since, using e.g.\ belief propagation, subtree ranking, infection eccentricity, rumour centrality, and the minimum description principle \cite{AltBra14, bubeck2017finding, PraVreFal12, ShaZam09, ZhuYing16}. These methods use only binary information about the vertices (infected vs not infected at some time $t>t_0$) as observational input. In an applied setting, a possibly noisy observation of the \emph{infection times} at a few predetermined \emph{sensor} vertices might be readily available, and with this extra information we might be able to detect the source by observing only a small subset of the nodes \cite{Pinto12,ZhuCheYin16}. With the exception of the recent work of \cite{Lecomte2020NoisySL}, not much is known about the number of required sensors in source detection if the spreading of the epidemic is very noisy. On the other hand, if we assume no noise in the spreading of the epidemic and the observations, the minimum number of sensors required to perfectly locate the source is equivalent to the MD problem \cite{zejnilovic2013network}. If in addition, the start time of the epidemic ($t_0$) is unknown, the minimum number of required sensors becomes equivalent to the \emph{double metric dimension} problem \cite{chen2014approximability}. The algorithmic aspects of the double MD in the source location context were investigated in \cite{CelPavSpiThi15, chen2014approximability}, and the double MD of Erd\H{o}s-R\'enyi random graphs was computed by \cite{SpiCelThi18}. Recently, \cite{OdoThi19} studied a version of the MD in Erd\H{o}s-R\'enyi graphs, where the sensors can be placed sequentially based on the observation times of previously placed sensors. 

\emph{Organisation of the rest of the paper.}
In Section \ref{s:model} we define the notions we use precisely, and we give the formula for the constant $c_{(\rho,\chi) }$ for general linear preferential attachment trees. In Section \ref{s:methods} we explain the general methodological background about embedding discrete trees in (continuous time) Crump-Mode-Jagers trees, and fringe trees and subtree properties. In Section \ref{s:proof} we prove our results. 
\section{Definitions and numerical values for $c_{(\rho,\chi)}$ }\label{s:model}

\begin{figure}
  \includegraphics[width=\textwidth]{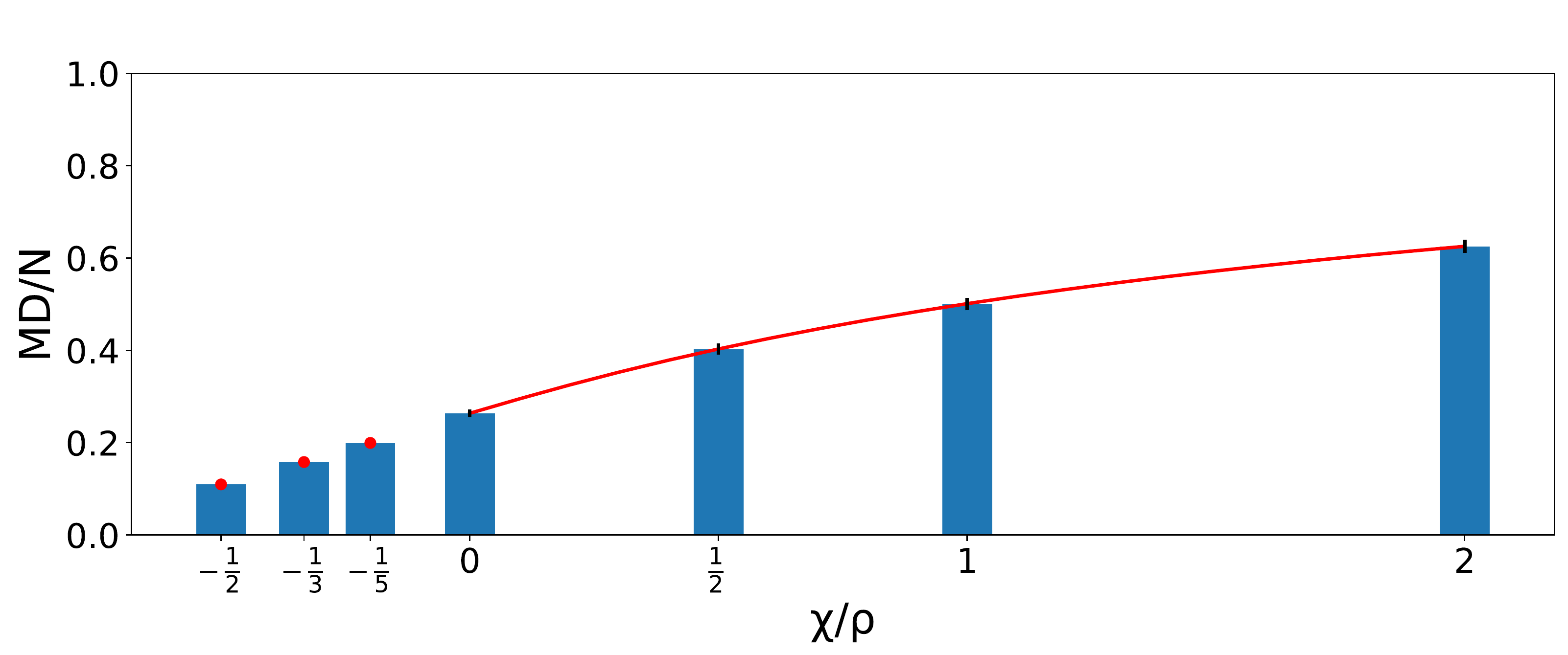}
   \captionsetup{width=\linewidth}
  \caption{The red dots and the red line show $c_{(\rho,\chi)}$ as a function of $\chi/\rho$ based on our theoretical results in Theorems \ref{thm:BST}, \ref{thm:RRT} and  \ref{thm:LPAT}. The blue bars show simulation results for $c_{(\rho,\chi)}$. We show the average of the normalized MD of 1000 independently simulated random trees with 1000 nodes. Unless they are too small to be visible on the plot, we also show the 95\% confidence intervals for the simulation results on top of the bar plots. }
  \label{fig:simulation}
\end{figure}

We start by giving a formal definition of the metric dimension.
\begin{definition}[MD]
\label{MD_def}
Let $G=(V,E)$ be a simple connected graph, and let us denote by $d(v,w) \in \mathbb{N}$ the length of the shortest path (that is, the number of edges) between nodes $v$ and $w$ that we call \emph{graph distance}. A subset $R\subseteq V$ is a \textit{resolving set} in $G$ if for every pair of nodes $u\ne v \in V$ there is a distinguishing node $w \in R$ for which $d(u,w)\ne d(v,w)$. The minimal cardinality of a resolving set is the \textit{metric dimension} (MD) of $G$, denoted by $\beta(G)$.
\end{definition}
The next definition helps us express the MD of fixed trees explicitly.
\begin{definition}[Leaves and exterior major nodes]
Let us denote by $\deg(v)$ the degree of a node $v\in V$. We say that a node $v \in V$ is a \textit{leaf} if $\deg(v)=1$, and it is a \textit{major node} if $\deg(v)\ge 3$. If a major node $v \in V$ has a path to a leaf that only contains degree-two vertices besides the beginning and the end of the path (i.e., a line-graph), we say that $v$ is an \textit{exterior major node}. Let us denote the set of leaves of $G$ by $L(G)$ and the set of exterior major nodes of $G$ by $K(G)$. 
\end{definition}
The following theorem characterises the metric dimension of a fixed tree. 
\begin{theorem}[Metric dimension of trees \cite{slater1975leaves}]
\label{thm_slater}
Consider a fixed tree $T$. If $T$ is a path graph, then $\beta(T)=1$. Otherwise,
\begin{equation}
\beta(T)=|L(T)|-|K(T)|.
\end{equation}
\end{theorem}
We refer the reader to  \cite{slater1975leaves} for a proper proof, but we explain the formula heuristically. It is not hard to see that if two or more leaves are attached to a major node by line-graphs, then the vertices at equal distance from the major node on these lines are indistinguishable by sensors that do not fall into these lines. Hence, all but one of the terminal leaves of such lines have to be sensors.

Now we state our more detailed results about families of trees growing according to general linear preferential attachment schemes, that is, we refine Theorem \ref{thm:meta} and express the limiting constant $c_{(\rho,\chi)}$ of the MD explicitly. Some of the numerical values acquired from the Theorems \ref{thm:BST}, \ref{thm:RRT} and \ref{thm:LPAT} below are shown in Figure \ref{fig:simulation} along with numerical approximation given by computer simulations.
  
\subsection*{Random binary search tree and $m$-ary increasing trees}

Recall that in an $m$-ary increasing tree is equivalent to a general linear preferential attachment tree with $\rho=m$ and $\chi=-1$, and that the for $m=2$, an $m$-ary increasing tree is equivalent to a random binary search tree. 

We write
\begin{equation}
\label{eq:gammadef}
\gamma(s,t)=\int_0^t x^{s-1} e^{-x} \diff x
\end{equation}
for the lower incomplete gamma function, and  
\begin{equation}
\label{eq:binomdef}
\binom{m}{i,j}=\frac{m!}{i!j!(m-i-j)!}
\end{equation}
for the generalized binomial coefficient.

\begin{theorem}[MD of $m$-ary increasing trees]
\label{thm:BST}
Let $(\mathcal{T}_n^{(m,-1)})_{n\ge 1}$ be a growing sequence of random $m$-ary increasing trees with $n$ vertices. Then,

\begin{equation}
\label{eq:BST}
\frac{\beta(\mathcal{T}_n^{(m,-1)})}{n} \xrightarrow{a.s. } \sum_{j=1}^m \frac{m-1}{(m-1+j)m^j}\binom{m}{j} + \sum_{\substack{0\le i+j \le m\\ i\ne 0 }} A_{i,j} \gamma \left(\frac{i+j}{m-1}+1, \frac{im}{m-1} \right),
\end{equation}
where for all $(i,j) \in \mathbb N^2$ with $i+j \le m$ and $(i,j)\ne (1,m-1)$
\begin{alignat}{2}
A_{i,j} &= \frac{(-1)^{i}}{m^{i+j}}\binom{m}{i,j} e^{\frac{im}{m-1}} \left(\frac{m-1}{im}\right)^{\frac{i+j}{m-1}+1},
\end{alignat}
except for $(i,j)= (1,m-1)$ we have
\begin{alignat}{2}
A_{1,m-1} &= \left(1-\frac{m}{m^{m}}\right)e^{\frac{m}{m-1}}\left(\frac{m-1}{m}\right)^{\frac{m}{m-1}+1} .
\end{alignat}
In particular, for the binary search tree ($m=2$), this expression evaluates to
\begin{equation}
\frac{\beta(\mathcal{T}_n^{(2,-1)})}{n} \xrightarrow{a.s.} \frac{ 3 e^4 - 48 e^2 +233}{384} \approx 0.1096868681.
\end{equation}
\end{theorem}

We provide two proofs to this theorem for $m=2$ below in Sections \ref{s:proof}: a combinatorial proof and a probabilistic proof. The probabilistic proof is more robust, and we are able to generalise that proof for $m>2$ and other types of attachments rules.

\subsection*{Random recursive tree}
\label{RRT}

As mentioned in the introduction, a random recursive tree is constructed by attaching each new node uniformly randomly to one of the existing nodes. It is also a special case of a general linear preferential attachment tree with parameters $\rho=1, \chi=0$.

\begin{theorem}[MD of random recursive trees]
\label{thm:RRT}
Let $\mathcal{T}_n^{(1,0)}$ be a sequence of random recursive trees with $n$ nodes. Then, 
\begin{equation}
\frac{\beta(\mathcal{T}_n^{(1,0)})}{n} \xrightarrow{a.s.} e\left(\int\limits_1^e x^{-1}e^{-x} \diff x+\gamma(2,1)\right)-1 \approx 0.263709059.
\end{equation}
\end{theorem}

\subsection*{Rich-get-richer trees}

Theorems \ref{thm:BST} and \ref{thm:RRT} covered general linear preferential attachment trees with $\chi \in \{-1, 0\}$. In the next theorem it suffices to state the result with $\chi=1$. These trees are often called rich-get-richer trees, as new nodes are more likely to attach to nodes with higher degrees. 

\begin{theorem}[MD of rich-get-richer trees]
\label{thm:LPAT}
Let $\mathcal{T}_n^{(\rho,1)}$ be a sequence of linear preferential attachment trees with $n$ nodes and $\chi=1$, $\rho>0$. Then, 
\begin{equation}\label{eq:thm_LPAT}
\begin{aligned}
\frac{\beta(\mathcal{T}_n^{(\rho,1)})}{n} &\xrightarrow{a.s. }-1+\int\limits_0^\infty (\rho+1)e^{-x(\rho+1)}\left(1+\frac{e^{ x+\frac{\rho}{\rho+1}(1-e^{-(\rho+1)x})} - e^{ x}}{\rho} \right)^{-\rho} \diff x \\
&\ \ \  \qquad + \int\limits_0^\infty (\rho+1)e^{-x(\rho+1)} e^{-\rho x+\frac{\rho}{\rho+1}(1-e^{-(\rho+1)x})} \diff x .
\end{aligned}
\end{equation}
\end{theorem}

The $\rho=\chi=1$ case corresponds to the positive linear preferential attachment tree, introduced by \cite{barabasi1999emergence}. For positive linear preferential attachment trees we can use Theorem \ref{thm:LPAT} and a numerical integration software \cite{WMathematica_Online} to obtain the following result result.

\begin{figure}
\small{
\begin{tabular}{|c||c|c|c|c|c|c|c|c|c|} \hline
$\chi/\rho$ & -1/2     & -1/3     & -1/4     & -1/5     & 0        & 1/2      & 1        & 2        & 10       \\ \hline
${\beta(\mathcal{T}_n^{(\rho,1)})}/{n}$          & 0.10969 & 0.15812 & 0.18377 & 0.19953 & 0.26371 & 0.40304 & 0.50120 & 0.62535 & 0.87501\\\hline
\end{tabular}}   \captionsetup{width=\linewidth}

  \caption{The table shows numerical values of the MD of general linear preferential attachment trees for some parameters $\chi/\rho$. The parameter values $-1/2, 0$ and $1$ correspond to the binary search tree, the random recursive tree, and the positive linear preferential attachment tree respectively.}
  \label{fig:table}
\end{figure}

\begin{corollary}
Let $\mathcal{T}_n^{(1,1)}$ be a sequence of positive linear preferential attachment trees with $n$ nodes. Then,
$$ \frac{\beta(\mathcal{T}^{(1,1)}_n)}{n} \xrightarrow{a.s. } c_{(1,1)} \approx -1+ 0.679824 +0.821372=0.501196.$$
\end{corollary}
 
\section{Method and discussion}\label{s:methods}
In this section we introduce fringe-trees and general results on their convergence, we explain the embedding of trees growing in discrete times into Crump-Mode-Jagers branching processes, and  relate the metric dimension to subtree properties.
\subsection{Fringe trees}

For the rest of the paper, all trees $T$ are considered to be rooted, which simply means that they have a special vertex denoted by $\mathrm{root}(T)$. In rooted trees, every vertex $v \in T \setminus \{ \mathrm{root}(T)\}$ has a parent, which is the first vertex on the path from $v$ to  $\mathrm{root}(T)$. For any vertex $v \in T$, let $T_v$ be the subtree of $T$ rooted at $v$, that is the connected subtree of $T$ that contains $v$ after removing the parent of $v$ (as a special case $T_{\mathrm{root}(T)}=T$). If we sample $v$ uniformly at random from $T$, we say that the random tree $T_v$ is a \textit{random fringe tree} of $T$. When $T$ is a deterministic tree, this definition is quite straightforward. However, we are interested in the case when $T$ itself is random, and in this case defining random fringe trees requires more care.

\begin{definition}
For rooted trees $S$ and $T$ let $n_S(T)$ be the cardinality of $\{ v \mid T_v = S\}$ and for a rooted tree property $\mathcal{P}$, let $n_\mathcal{P}(T)$ be the cardinality of $\{ v \mid T_v \in \mathcal{P}\}$.
\end{definition}

When $T$ is deterministic $n_S(T)/|T|$ defines the random fringe tree distribution. When $\mathcal{T}$ is random, we can think of the sampling of $\mathcal{T}$ and $v$ as a combined random event, which again gives rise to a distribution over trees. This is called the annealed fringe tree distribution. In this paper, we are interested in the \emph{quenched fringe tree distribution}. In the quenched version, we think of $n_S(\mathcal{T})/|\mathcal{T}|$ as a distribution that is itself random. Since we are interested in the convergence of fringe tree distributions as the size of the trees tend to infinity, we are going to focus on the convergence of the random variables $n_S(\mathcal{T}_n)/|\mathcal{T}_n|$ (almost surely (a.s.) or in probablity (p)).

We also defined the seemingly more general notion of $n_\mathcal{P}(T)$, however, in our applications whenever we can say something about the convergence of $n_S(\mathcal{T}_n)/|\mathcal{T}_n|$, we have a similar result for $n_\mathcal{P}(\mathcal{T}_n)/|\mathcal{T}_n|$. In fact, since working with subtree properties will be very convenient for computing the MD (see Lemma \ref{lem:PE}), we only state the results from the fringe tree literature on $n_\mathcal{P}(T)$.

\begin{theorem}[\cite{aldous1991asymptotic}, Theorem 1.2 of \cite{janson2013asymptotic}]
\label{thm:GW_helper}
Let $\mathcal{GW}_n$ be a sequence of Galton-Watson trees conditioned to have $n$ vertices, with offspring distribution $\xi$, where $\E[\xi]=1$ and $\E[\xi^2]<\infty$. Let $\mathcal{F}$ be the \emph{unconditioned} Galton-Watson tree with the same offspring distribution. Then, for every subtree property $\mathcal{P}$,
\begin{equation}
\label{eq:GW_helper}
\frac{n_\mathcal{P}(\mathcal{GW}_n)}{n} \xrightarrow{p} \P(\mathcal{F} \in \mathcal{P}).
\end{equation}
\end{theorem}

The previous theorem applied to any Galton-Watson tree with $\E[\xi]=1$ and $\E[\xi^2]<\infty$. The next theorem only applies to a single family of growing trees, the binary search tree. We will use it to give a combinatorial proof of the LLN of the MD of binary search trees (second part of Theorem \ref{thm:BST}).

\begin{theorem}[\cite{aldous1991asymptotic}]
\label{thm:BST_helper}
Let $(\mathcal{T}_n^{(2,-1)})_{n\ge 1}$ be a growing sequence of binary search trees of size $n$. Then, for every subtree property $\mathcal{P}$,
\begin{equation}
\label{eq:BST_helper}
\frac{n_\mathcal{P}(\mathcal{T}_n^{(2,-1)})}{n} \xrightarrow{p} \sum\limits_{k=1}^\infty \frac{2}{(k+1)(k+2)}\P(\mathcal{T}_k^{(2,-1)} \in \mathcal{P}).
\end{equation}
\end{theorem}
In words, this theorem says that the fringe-tree distribution of a random binary search tree is again a random binary search tree with a \emph{random size}: the probability that the size of the fringe-tree is $k$ is ${2}/{((k+1)(k+2))}$.
A similar statement can be made for random recursive trees, however, we do not include this statement as it will not be used in our proofs. Instead we introduce a more powerful theorem which will help to strengthen the convergence to almost sure, treat $m$-ary increasing trees for general $m\ge 2$, random recursive trees, and linear preferential attachment trees.

\subsection{Crump-Mode-Jagers trees and fringe trees}
\label{subsec:CMJ}

A \textit{Crump-Mode-Jagers} (CMJ) branching process generalizes, among many other random tree models, $m$-ary increasing trees and random recursive trees. Heuristically speaking, CMJ branching processes  provide a method of embedding trees growing in discrete steps into a corresponding  continuous time process. The CMJ process is defined by a point process $\Xi=(\xi_1, \xi_2, \dots)$, called the reproduction process.  At time zero, a single vertex is born, which becomes the root of the tree, and the children of the root are born at time $\xi_1, \xi_2, \cdots$. Similarly, each vertex $v$ born at time $t_v$ has an independent copy of $\Xi$ denoted as $\Xi_v=\xi_{v,1}, \xi_{v,2}, \cdots$, and the offspring of $v$ are born at time $t_v+\xi_{v,1}, t_v+\xi_{v,2}, \cdots$. So far we defined a branching process that grows over time. We obtain a random tree from this branching process by stopping the process at time $\tau$ and taking only the vertices (individuals) that have already been born. The stopping time $\tau$ can depend on the tree (very often $\tau$ is the time the $n^{th}$ individual is born), or it can be an independent random variable.

\begin{definition}[Linear preferential attachment reproduction process]
\label{def:LPARP}
Let the reproduction process $\Sigma_{\rho,\chi}=(\xi_1, \xi_2, \dots, )$ with parameters $\rho>0$ and $\chi \in\{-1,0,1\}$ be a \textit{linear preferential attachment reproduction process} if
\begin{equation}\label{eq:exp-repr}\xi_j-\xi_{j-1} \sim \mathrm{Exp}(\rho+\chi (j-1))\end{equation}
are independent exponential random variables, with the convention that $\xi_0:=0$ (but it does not count as a birth event). If $\chi=-1$, let us also assume $\rho \in \mathbb N$ and let us truncate the process to $\rho $ terms (i.e. $\Sigma_{\rho,\chi}=(\xi_1,\dots, \xi_{\rho-1}, \xi_\rho)$), which assures that the exponential random variables in \eqref{eq:exp-repr} are well-defined.
\end{definition}

\begin{lemma}\label{lem:embedding}
A CMJ tree with a linear preferential attachment reproduction process $\Sigma_{\rho,\chi}$ stopped when it reaches $n$ vertices has the same distribution as a linear preferential attachment tree with $n$ vertices and parameters $\rho$ and $\chi$.\end{lemma}

This lemma is due to the memoryless property of exponential random variables; the proofs can be found in \cite[Sections 6.3, 6.4]{holmgren2017fringe}. 

The interesting property of CMJ trees is that the fringe tree distribution of the random CMJ tree stopped at $n$ vertices is again a random CMJ tree, with the \emph{same reproduction process},   stopped at a random time that is independent of the number of vertices. This independence of the stopping time will be heavily exploited in our proofs. In this paper, we only use the results on the fringe trees of linear preferential attachment trees. We refer to \cite{holmgren2017fringe} for the general statement on CMJ trees.

\begin{theorem}[\cite{jagers1975branching, nerman1981convergence},Theorem 5.14 of \cite{holmgren2017fringe}]
\label{thm:CMJ_helper}
Let $(\mathcal{T}_n^{(\rho,\chi)})_{n\ge 1}$ be a growing sequence of linear preferential attachment trees with $n$ vertices and parameters $\rho>0$ and $\chi\in\{-1,0,1\}$. Let $\mathcal{F}$ be the corresponding CMJ tree stopped at random time $\mathrm{Exp}(\rho+\chi)$. Then, for every subtree property $\mathcal{P}$,
\begin{equation}
\label{eq:CMJ_helper}
\frac{n_\mathcal{P}(\mathcal{T}_n^{(\rho,\chi)})}{n} \xrightarrow{a.s.} \P(\mathcal{F} \in \mathcal{P}).
\end{equation}
\end{theorem}

\subsection{Expressing the metric dimension with subtree properties}
In this section we reduce the metric dimension of trees to counting subtrees with certain properties. Recall Theorem \ref{thm_slater} that expresses the MD of a tree as the difference between the number of leaves and that of exterior major vertices.
\begin{definition}
\label{PE}
Let $\mathcal{P}_L$ be the subtree property that the subtree is a single vertex, that is a leaf. Let $\mathcal{P}_K$ be the subtree property that the root has degree at least \emph{two} and at least one of its subtrees is a line-graph to a leaf (a single vertex is considered to be a line). 
\end{definition}

\begin{figure}[h]
  \includegraphics[width=0.9\textwidth]{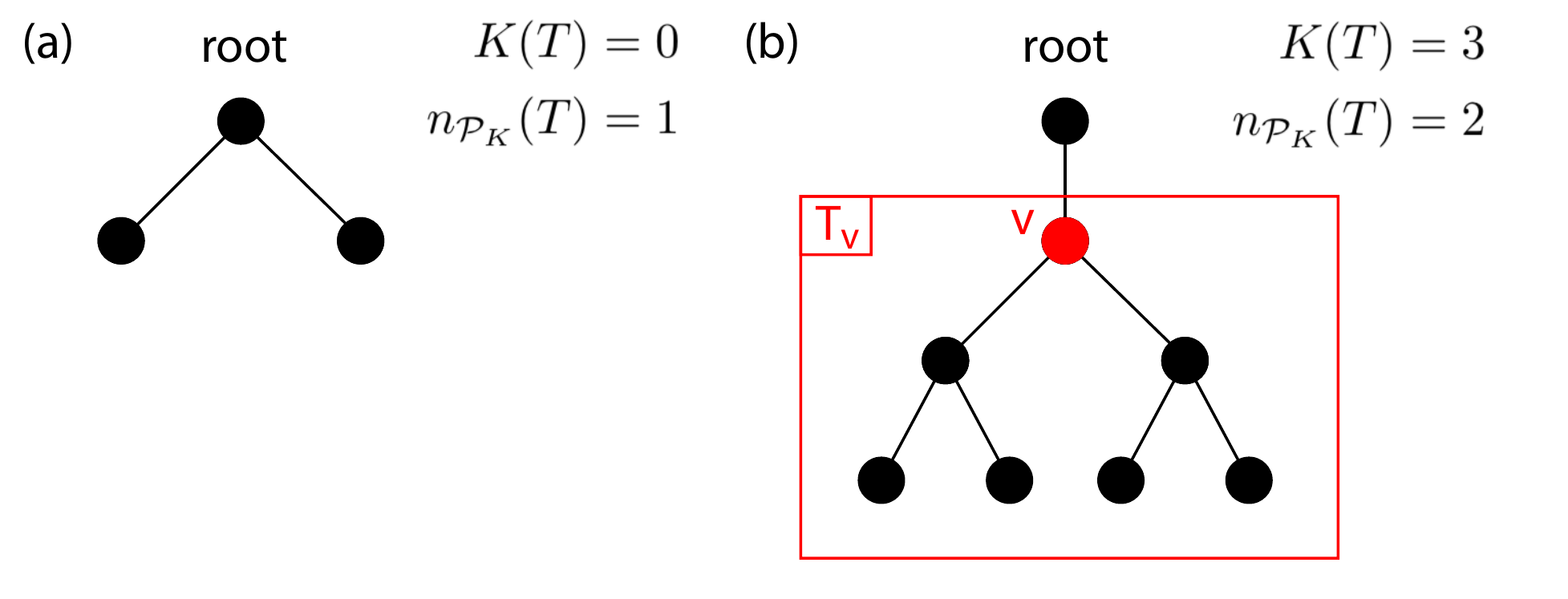}
   \captionsetup{width=\linewidth}
  \caption{Illustration for the proof of Lemma \ref{lem:PE}. The subfigures (a) and (b) show the smallest trees where $n_{\mathcal{P}_K}(T)-K(T)=\pm 1$, respectively. The inequality $n_{\mathcal{P}_K}(T)-K(T)>0$ holds only for trees in which the root has degree 2, and the root has a line-graph to a leaf. In this case the root has property $\mathcal{P}_K$, but it does not count into $K(T)$ since it has degree 2. The inequality $n_{\mathcal{P}_K}(T)-K(T)<0$ holds only for trees in which the root that has degree 1, and the first descendant of the root with degree 3 (node $v$) has no other line-graph to a leaf. In this case $v$ counts into $K(T)$, but it does not have property $\mathcal{P}_K$.}
  \label{fig:lem32}
\end{figure}

\begin{lemma}
\label{lem:PE}
For any sequence of trees $T_n$, with $|T_n| \rightarrow \infty$ and  $\mathcal{P}_L,\mathcal{P}_K$ given by Definition \ref{PE},
\begin{equation}
\frac{\beta(T_n)}{|T_n|} = \frac{|L(T_n)| - |K (T_n)|}{|T_n|} = \frac{n_{\mathcal{P}_L}({T}_n)}{|T_n|}-\frac{n_{\mathcal{P}_K}({T}_n)}{|T_n|}+\frac{\varepsilon}{|T_n|},
\end{equation}
where $\varepsilon\in \{-1,0,1\}$.
\end{lemma}

\begin{proof}
We are going to show the equivalent statement that for any deterministic rooted tree $T$, we must have $n_{\mathcal{P}_L}(T)=~|L(T)|$ and $|n_{\mathcal{P}_K}(T)-|K(T)|| \le 1$. The equality $n_{\mathcal{P}_L}(T)=L(T)$ follows from the definition. Next we show that $|n_{\mathcal{P}_K}(T)-|K(T)|| \le 1$ (see also Figure \ref{fig:lem32}).

If $v \in V$ is not the root of $T$, then $T_v \in \mathcal{P}_K$ implies $v \in K(T)$. This is because $v$ must have at least two children by the property $\mathcal{P}_K$ and a parent vertex since $v$ is not the root, which means that $v$ has degree at least three. By the definition of $\mathcal{P}_K$, $T_v$  contains a line-graph to a leaf. Hence $n_{\mathcal{P}_K}(T)-1 \le |K(T)|$.

For the other direction, we argue that $v \in K(T)$ implies $T_v \in \mathcal{P}_K$, except for at most one vertex $v \in V$. This is because $v$ has degree at least three by the exterior major vertex property, two of which must be the children of $v$ in $T_v$. Moreover, the path of degree two vertices to a leaf ensured by the exterior major vertex property must be a subtree that is a path in $T_v$, unless the path of degree two vertices to a leaf is through the parent of $v$. This can only happen if all ancestors of $v$ have degree two, the $\mathrm{root}(T)$ has degree one or two, and if $\mathrm{root}(T)$ has another subtree that does not contain $v$, this must be a line-graph. In other words, $\mathrm{root}(T)$ can have only one subtree with a major vertex, and $v$ must be the first major vertex on this subtree, if such a $v$ exists. Hence $|K(T)| -1 \le n_{\mathcal{P}_K}(T)$.
\end{proof}

In all of our proofs we will combine Lemma \ref{lem:PE} with either Theorem \ref{thm:GW_helper}, \ref{thm:BST_helper} or \ref{thm:CMJ_helper}.  Since $\P(\mathcal{F} \in \mathcal{P}_L)$ is an easy computation in all cases, most of the difficulty will come from computing $\P(\mathcal{F} \in \mathcal{P}_K)$, where $\mathcal{F}$ is a random tree having the limiting fringe tree distribution (see formulas \eqref{eq:GW_helper}, \eqref{eq:BST_helper} and \eqref{eq:CMJ_helper}). To compute $\P(\mathcal{F} \in \mathcal{P}_K)$, often it will be useful to condition on the degree of the root of $\mathcal{F}$, and another event $\mathcal E$, that will be the ringing time of the doomsday clock $\mathrm{Exp}(\rho+\chi)$ in Theorem \ref{thm:CMJ_helper}. Recall that for any non-negative discrete random variable $Y$ we denote by 
$$G_Y(x)=\sum_{n=0}^\infty \P(Y=n) x^n$$
the probability generating function of $Y$ evaluated at $x$.

\begin{lemma}
\label{lem:PFmidX}
Let $\kappa$ be the degree of $\mathrm{root}(\mathcal{F})$. If $v$ is an offspring of $\mathrm{root}(\mathcal{F})$, let $B_v$ be the event that $\mathcal{F}_v$ is a line-graph. Suppose that for some event $\mathcal E$ the indicators of $B_v$, conditioned on $\kappa$ and $\mathcal{E}$, are independent and identically distributed Bernoulli random variables with parameter $q$. Then,
\begin{equation}
\label{eq:PFmidX}
\P(\mathcal{F} \in \mathcal{P}_K \mid \mathcal{E}) = 1-G_{\kappa \mid \mathcal{E}}(1-q)-q\P(\kappa=1 \mid \mathcal{E}).
\end{equation}
\end{lemma}

\begin{proof}
 Let $A$ be the event that the root has at least two offspring, and $B_i$ be the event that the root of the $i^{th}$ subtree is born, and the subtree is a line-graph. Let us denote the event $B:=\cup_{i\ge 1}B_i$. By definition the event $\mathcal{F} \in \mathcal{P}_K = A\cap B$. Then we can write,
\begin{align}\label{eq:pe-calc}
\P(\mathcal{F} \in \mathcal{P}_K \mid \mathcal{E}) &= \P(A \cap B \mid \mathcal{E})=1- \P(A^c \cup B^c \mid \mathcal{E}) \nonumber \\
&= 1-\left(\P(\kappa =0 \mid \mathcal{E})+\P(\kappa =1 \mid \mathcal{E})+\sum\limits_{k=2}^\infty \P(B^c_1 \cap \dots \cap B^c_k \mid \kappa =k, \mathcal{E}) \P(\kappa =k \mid \mathcal{E})\right)\nonumber \\
&= 1-\P(\kappa =0 \mid \mathcal{E})-\P(\kappa =1 \mid \mathcal{E})-\sum\limits_{k=2}^\infty (1-q)^k \P(\kappa =k \mid \mathcal{E}),
\end{align}
where the last line followed since we assumed that $B_i$ are independent $\mathrm{Ber}(q)$ conditioned on $\kappa$ and $\mathcal{E}$.
Noticing that the last sum is the generating function of $(\kappa \mid \mathcal{E})$ evaluated at $1-q$, except that the index starts from two instead of zero, we get
\begin{align}
\P(\mathcal{F} \in \mathcal{P}_K \mid \mathcal{E}) &= 1-\P(\kappa =0 \mid \mathcal{E})-\P(\kappa =1 \mid \mathcal{E})-G_{\kappa \mid \mathcal{E}}(1-q)+\sum\limits_{n=0}^1 (1-q)^k \P(\kappa =k \mid \mathcal{E}) \nonumber \\
& = 1-G_{\kappa \mid \mathcal{E}}(1-q)-q\P(\kappa=1 \mid \mathcal{E})
\end{align}
\end{proof}

\begin{remark}
If we were interested in simply exterior vertices, using the same ideas, the expression in equation \eqref{eq:PFmidX} would simplify to $1-G_{\kappa \mid\mathcal{E}}(1-q)$.
\end{remark}

\section{Proofs}\label{s:proof}
In this section we prove Theorems \ref{thm:GW}, and \ref{thm:BST}--\ref{thm:LPAT}.
\subsection{Metric dimension of conditioned Galton-Watson trees}

\begin{proof}[Proof of Theorem \ref{thm:GW}]
Combining Lemma \ref{lem:PE}, and Theorem \ref{thm:GW_helper}, we have that
\begin{equation}
\label{eq:GW1}
\frac{\beta(\mathcal{T}_n)}{n} \xrightarrow{p} \P(\mathcal{F} \in \mathcal{P}_L) - \P(\mathcal{F} \in \mathcal{P}_K),
\end{equation} 
where $\mathcal{F}$ is a Galton-Watson tree with offspring distribution $\xi$.

Clearly, $\P(\mathcal{F} \in \mathcal{P}_L) = p_0$. It remains to compute $\P(\mathcal{F} \in \mathcal{P}_K)$. Since the subtrees of each offspring in a Galton-Watson tree are independent the conditions of Lemma \ref{lem:PFmidX} are satisfied without conditioning.

We still need to find the value of $q=\P(B_v)$, which is the probability that $\mathcal{F}_v$ is a line-graph since the subtree $\mathcal{F}_v$ is independent of the degree of the root of $\mathcal{F}$. Vertex $v$ can have (i) zero offspring, in which case $\mathcal{F}_v$ is a (trivial) line graph, (ii) one offspring, in which case $\mathcal{F}_v$ is a line with probability $q$, or (iii) more than one offspring, in which case $\mathcal{F}_v$ is not a line. Hence, we have the equation
\begin{equation}
q = p_0+qp_1,
\end{equation}
which gives $q=p_0/(1-p_1)$. Substituting equation \eqref{eq:PFmidX} into equation \eqref{eq:GW1} with $q=p_0/(1-p_1)$ we obtain the desired result.
\end{proof}

\subsection{Metric dimension of binary search trees (combinatorial proof)}

\begin{proof}[Proof of Theorem \ref{thm:BST}, $m=2$]

Combining Lemma \ref{lem:PE} and Theorem \ref{thm:BST_helper}, we obtain that
\begin{equation}
\label{eq:BST1}
\frac{\beta(\mathcal{T}_n^{(2,-1)})}{n} \xrightarrow{p} \sum\limits_{k=1}^\infty \frac{2}{(k+1)(k+2)}\P(\mathcal{T}_k^{(2,-1)} \in \mathcal{P}_L) - \sum\limits_{k=1}^\infty \frac{2}{(k+1)(k+2)}\P(\mathcal{T}_k^{(2,-1)} \in \mathcal{P}_K)
\end{equation} 

Clearly $\P(\mathcal{T}^{(2,-1)}_k \in \mathcal{P}_L)$ equals $1$ for $k=1$ and $0$ for $k>1$, which implies that the first term in equation \eqref{eq:BST1} is $1/3$.

It remains to compute the second term in equation \eqref{eq:BST1}. Recall full and potential vertices from the description of binary search trees on page 2. Let $k'=k-1$ and $S_k \in \{0,\dots, k'\}$ be the number of (full) vertices in the left subtree when the tree has $k$ (full) vertices. Notice, that the number of potential vertices in the left and right subtrees follows a P\'olya urn process with two urns initially with a single white and a single black ball, and that the number of full vertices is always one less than the number of potential vertices in each subtree.  Elementary calculation using induction shows that $S_k$ is then uniform over the set $\{0,\dots, k'\}$, or in other words   $\P(S_k=\ell)=1/(k'+1)$,  see e.g. \cite[Theorems 5.2, 5.3]{Hofs17b}.

Since $S_k \in \{0,k'\}$ implies that the root has degree less than two $\P(\mathcal{T}_k \in \mathcal{P}_K | S_k \in  \{0,k'\} )=0$. By the law of total probability, 
\begin{equation}
\begin{aligned}
\label{eq:BST_ltp}
\P(\mathcal{T}^{(2,-1)}_k \in \mathcal{P}_K)&=\sum\limits_{\ell=0}^{k'}\P(\mathcal{T}^{(2,-1)}_k \in \mathcal{P}_K | S_k=\ell)\P(S_k=\ell)\\
&=\frac{1}{k'+1}\sum\limits_{\ell=1}^{k'-1}\P(\mathcal{T}_k^{(2,-1)} \in \mathcal{P}_K | S_k=\ell).
\end{aligned}
\end{equation}
Now we focus on the second condition of $\mathcal{P}_K$, the existence of a subtree that is a line. If a subtree has $\ell$ vertices, we argue that the probability that it is a line is $$\prod\limits_{i=3}^\ell \frac{2}{i}=\frac{2^{\ell-1}}{\ell!}.$$
Indeed, if the subtree has just one or two vertices, it must be a line. Thereafter, conditionally that the subtree is a line after having $i-1$ vertices, when we place the $i^{th}$ vertex into the subtree, we have to sample from $i$ possible places, only two of which keep the subtree a line. Namely, the children of the last vertex on the line. Here we use that the placement of vertices in the binary search tree is uniform over the possible locations, and conditioned that the vertex falls into the left (resp. right) subtree, its placement is uniform over the available locations within this subtree.  
To compute the probability that at least one of the subtrees is a line we apply an elementary inclusion-exclusion argument. For $1\le \ell \le k'-1$, we have
\begin{align}
\P(\mathcal{T}_k^{(2,-1)} \in \mathcal{P}_K | S_k=\ell)&=\P(B_1 | S_k=\ell)+ \P(B_2 | S_k=\ell) -\P(B_1 \cap B_2 | S_k=\ell) \nonumber \\
&=\frac{2^{\ell-1}}{\ell!} + \frac{2^{k'-\ell-1}}{(k'-\ell)!} -\frac{2^{k'-2}}{\ell!(k'-\ell)!},
\end{align}
where in the last term we used that conditioned on their sizes, the left and right subtree evolve independently.
Substituting the rhs back into equation \eqref{eq:BST_ltp} and using the basic identities of binomial coefficients, and recalling that $k'=k-1$, we obtain
\begin{align}
\label{eq:BST_bin}
\sum\limits_{\ell=1}^{k'-1}\P(\mathcal{T}^{(2,-1)}_k \in \mathcal{P}_K | S_k=\ell) &= \sum\limits_{\ell=1}^{k'-1}  \frac{2^{\ell-1}}{\ell!} + \frac{2^{k'-\ell-1}}{(k'-\ell)!} -\frac{2^{k'-2}}{\ell!(k'-\ell)!}=\sum\limits_{\ell=1}^{k'-1}  \frac{2^{\ell}}{\ell!} - \frac{2^{k'-2}}{k'!}\sum\limits_{\ell=1}^{k'-1} \binom{k'}{\ell} \nonumber \\
&=\sum\limits_{\ell=1}^{k'-1}  \frac{2^{\ell}}{\ell!}-\frac{2^{k'-2}(2^{k'}-2)}{k'!}.
\end{align}
Substituting \eqref{eq:BST_bin} into \eqref{eq:BST_ltp} and then into \eqref{eq:BST1} we obtain (with $k'=k-1$)
\begin{equation}
\label{eq:BST_intermediate}
\frac{\beta(\mathcal{T}_n^{(2, -1)})}{n} \xrightarrow{p} \frac{1}{3} - \sum\limits_{k=3}^\infty \frac{2}{(k+1)(k+2)(k'+1)} \left( \sum\limits_{l=1}^{k'-1}  \frac{2^{l}}{l!}-\frac{2^{k'-2}(2^{k'}-2)}{k'!}\right).
\end{equation}
Getting a closed form expression for $\sum_{\ell=1}^{k'-1}  2^{\ell}/\ell!$ is difficult, but it is clearly bounded by $e^2$. Since the sum $\sum_{k=\ell+2}^\infty 2/(k(k+1)(k+2))$ is also bounded, we can swap the order of the sums to get the easier expression
\begin{align}
\label{eq:BST_firstterm}
\sum\limits_{k=3}^\infty \frac{2}{k(k+1)(k+2)} \sum\limits_{\ell=1}^{k-2}  \frac{2^{\ell}}{\ell!} = \sum\limits_{\ell=1}^{\infty}  \frac{2^{\ell}}{\ell!} \sum\limits_{k=\ell+2}^\infty \frac{2}{k(k+1)(k+2)}. 
\end{align}
The sum $\sum_{k=\ell+2}^\infty 2/(k(k+1)(k+2))$ can be evaluated by elementary arithmetic operations and a telescopic sum. Indeed,
\begin{align}
\sum\limits_{k=\ell+2}^\infty \frac{2}{k(k+1)(k+2)}=\sum\limits_{k=\ell+2}^\infty \left(\frac{1}{k(k+1)}-\frac{1}{(k+1)(k+2)}\right) = \frac{1}{(\ell+2)(\ell+3)}. 
\end{align}
Substituting back into equation \eqref{eq:BST_firstterm}, elementary arithmetic operations give
\begin{align}
 \sum\limits_{\ell=1}^{\infty}  \frac{2^{\ell}}{\ell!} \frac{1}{(\ell+2)(\ell+3)} = \sum\limits_{\ell=1}^{\infty} \frac{(\ell+1)2^{\ell}}{(\ell+3)!} 
&=\sum\limits_{\ell=1}^{\infty} \frac{(\ell+3)2^{\ell}}{(\ell+3)!} - \sum\limits_{\ell=1}^{\infty} \frac{2\cdot 2^{\ell}}{(\ell+3)!} =\frac13.
\end{align}
The last equality follows if we notice that the sum that we are subtracting is the same as the sum we are subtracting from, except it is shifted by one index. Hence, the result of the subtraction is the simply the first term of the sum. A similar compuation yields the following equalities,
\begin{align}
\sum\limits_{k=3}^\infty \frac{2}{k(k+1)(k+2)}\frac{2^{k'-2}(2^{k'}-2)}{k'!} =\sum\limits_{k=3}^\infty \frac{2^{2k-3}}{(k+2)!}-\frac{2^{k-1}}{(k+2)!} =\frac{ 3 e^4 - 48 e^2 +233}{384}.
\end{align}
Finally, substituting into equation \eqref{eq:BST_intermediate} we obtain
\begin{equation}
\frac{\mathrm{\beta}(\mathcal{T}_n^{(2,-1)})}{n} \xrightarrow{p} \frac13-\frac13+ \frac{ 3 e^4 - 48 e^2 +233}{384} = \frac{ 3 e^4 - 48 e^2 +233}{384},
\end{equation}
which is the desired result.
\end{proof}

\subsection{Metric dimension of general linear preferential attachment trees (proof using fringe trees)}
In this section we prove Theorems \ref{thm:BST}, \ref{thm:RRT} and \ref{thm:LPAT}. First, we state a few preliminary lemmas. We handle all values of $(\rho, \chi)$ together until the last step when we obtain the numerical values.
Recall that Lemma \ref{lem:embedding} gives an embedding of $(\mathcal T^{(\rho,\chi)}_n)_{n\ge 1}$ into a Crump-Mode-Jagers process with reproduction function $\Sigma_{\rho,\chi}$ given in Definition \ref{def:LPARP}. Combining Lemma \ref{lem:PE}, and Theorem \ref{thm:CMJ_helper}, we have that
\begin{equation}
\label{eq:BST_p_1}
\frac{\beta(\mathcal{T}_n^{(\rho,\chi)})}{n} \xrightarrow{a.s.} \P(\mathcal{F} \in \mathcal{P}_L) - \P(\mathcal{F} \in \mathcal{P}_K),
\end{equation} 
where $\mathcal{F}$ is a CMJ tree with offspring point process $\Sigma_{\rho,\chi}$ stopped at random time $\tau=\mathrm{Exp}(\rho+\chi)$.

By Definition \ref{def:LPARP}, the time of the first offspring of the root of $\mathcal{F}$ is an $\mathrm{Exp}(\rho)$ random variable. To find $\P(\mathcal{F} \in \mathcal{P}_L)$ we need to compute the probability that the doomsday clock $\mathrm{Exp}(\rho+\chi)$ rings before the first offspring clock $\mathrm{Exp}(\rho)$. Hence, 
\begin{equation}\label{eq:leaf}
\P(\mathcal{F} \in \mathcal{P}_L) = \frac{\rho+\chi}{2\rho+\chi}.
\end{equation}

Next, we check that the conditions of Lemma \ref{lem:PFmidX} are satisfied, which will help us to find $\P(\mathcal{F} \in \mathcal{P}_K)$.
Let $\Sigma_{\rho,\chi}=(\xi_1, \xi_2 \dots)$ be a linear preferential attachment reproduction process as described in Definition \ref{def:LPARP}. 
We will apply Law of Total Probability with respect to the ringing time of the doomsday clock $\tau$. So, 
for infinitesimal $\mathrm dx$, let us take $\mathcal E_x:=\{\tau\in(x, x+\mathrm dx)\}$ be the event that the doomsday clock $\tau$ rings in the interval  $(x, x+\mathrm dx)$. Recall that we denote by $\kappa$ the degree of the root of $\mathcal F$. Recall that we write $\kappa$ for the number of children of the root in the limiting fringe tree $\mathcal F$. 
\begin{lemma}
\label{lem:can_use}
Conditioned on $\mathcal E_x\cap \{\kappa=k\}$,  the (unordered) set of times  $\{ \xi_1, \dots, \xi_k\} $ have the same distribution as $k$ i.i.d. random variables with density
\begin{equation}
\label{eq:def_gy}
g_x(y)=\frac{1}{Z_{g}(x)}e^{\chi y}
\end{equation}
supported on the interval $[0,x]$, with $Z_{g}(x)=\int_0^x e^{\chi y} \diff y$.
\end{lemma}
This statement is commonly known for $\chi=0$, when $\Sigma_{(\rho,0)}$ is a Poisson point process (PPP) on $\R^+$ with intensity $\rho$. In this  case, the lemma states that conditioned on the event that $\Sigma_{(\rho,0)}$ has $k$ points on the interval $[0,x]$, the locations of these points have the same distribution as that of $k$ i.i.d. uniform random variables on $[0,x]$. 
\begin{proof}[Proof of Lemma \ref{lem:can_use}]
Recall the distribution of the consecutive birth times 
$$\xi_j-\xi_{j-1}\ {\buildrel d \over =}\ \mathrm{Exp}(\rho+(j-1)\chi).$$ 
Conditioned on $\mathcal E_x$, the density that there are $k$ children of the fringe-root, precisely born at ordered times $\underline r:=(r_1, r_2, \dots, r_k)$, and the $(k+1)$-st child has $r_{k+1}>x$ is:
\[  f_o(k, r_1, \dots, r_k \mid \mathcal E_x):=\rho e^{-\rho r_1} (\rho+\chi) e^{-(\rho+\chi) (r_2-r_1) } \cdot \dots \cdot (\rho+\chi (k-1)) e^{-(\rho+\chi (k-1)) (r_k-r_{k-1}) } e^{-(\rho+\chi k) (x-r_k)}.\]
Observing that the coefficient of $r_j$ in the exponent is $\chi$, we see that
\begin{equation}
\label{eq:dens}  f_o(k, r_1, \dots, r_k \mid \mathcal E_x)=\frac{1}{Z_{f_o}(x) } \cdot e^{\chi(r_1 + \dots + r_k)} = \frac{Z_{g}(x)^k}{Z_{f_o}(x)}\prod\limits_{i=1}^k g_x(r_i),\end{equation}
where $Z_{f_o}(x)=\mathrm{e}^{-(\rho+\chi k)}/\prod_{i=0}^{k-1} (\rho+i\chi)$ is the normalizing factor independent of $\underline r$ (as long as $\underline r$ is really an ordered sequence, otherwise $f_o(k, r_1, \dots, r_k \mid x)=0$). However, we are not interested in the density of the ordered set of times. The \textit{unordered} set of times $\{\xi_1, \dots \xi_\kappa\}$ has density
\[ f_u(k, r_1, \dots, r_k \mid \mathcal E_x) = \frac{1}{k!} f_o(k, r_1, \dots, r_k \mid \mathcal E_x) = \frac{Z_{g}(x)^k}{k! Z_{f_o}(x)} \prod\limits_{i=1}^k g_x(r_i)\]
by the symmetry of the possible permutations of $r_1, \dots r_k$. Conditioning on $k$, by Bayes rule we know that 

\[ f_u(r_1, \dots, r_k \mid \mathcal E_x,\kappa =k) = \frac{1}{\P(\kappa =k\mid \mathcal E_x)}f_u(k, r_1, \dots, r_k \mid \mathcal E_x) = \frac{1}{Z_{f_u}(x)} \prod\limits_{i=1}^k g_x(r_i),  \]
where $Z_{f_u}(x)$ is the appropriate normalizing factor independent of $\{ r_1, \dots, r_k \}$, that is $Z_{f_u}(x)=Z_g(x)^k$. Since the density $f_u(r_1, \dots, r_k \mid \mathcal E_x,\kappa =k)$ is the product of the densities $g(r_i)$, the random variables $\{\xi_1, \dots \xi_\kappa\}$ must be i.i.d., with density $g_x(y)$.
\end{proof}
The implication of this lemma is that conditioned on $\mathcal E_x$ and $\kappa=k$, the $k$ subtrees of the fringe root are born independently at times following density $g_x(y)$, and evolve independently. Consequently, we can apply Lemma \ref{lem:PFmidX}, and we proceed to computing the terms that appear in \eqref{eq:PFmidX}. Some of these terms can be simply deduced from a result of \cite{holmgren2017fringe}.

\begin{lemma}[Theorem A.7.\ of \cite{holmgren2017fringe}]\label{lem:generating}
The offspring distribution of the root (denoted by $\kappa$) of a linear preferential attachment tree with parameters $\rho$ and $\chi$ stopped at time $x$ is given by $\mathrm{NBin}(\rho, e^{x})$ if $\chi=1$, $\mathrm{Poi}(\rho x)$ if $\chi=0$ and $\mathrm{Bin}(\rho,1-e^{-x})$ if $\chi=-1$, where $\mathrm{NBin}$ denotes the negative binomial distribution, $\mathrm{Poi}$ denotes the Poisson distribution and $\mathrm{Bin}$ denotes the binomial distribution. In particular,
\begin{equation}
\label{eq:gen_fn}
G_{\kappa \mid \mathcal E_x}(z)=
\begin{cases}
(e^{\chi x} + (1-e^{\chi x})z)^{-\rho/\chi} & \text{for $\chi = \pm 1$} \\
e^{-x(1-z)} & \text{for $\chi = 0, \rho=1$},
\end{cases}
\end{equation}
and
\begin{equation}
\label{eq:P1}
\P(\kappa=1 \mid  \mathcal E_x)=
\begin{cases}
-\frac{\rho}{\chi} (1-e^{\chi x})e^{-x(\rho+\chi)} & \text{for $\chi = \pm 1$} \\
xe^{-x} & \text{for $\chi = 0, \rho=1$}.
\end{cases}
\end{equation}
\end{lemma}
We refer the reader to \cite{holmgren2017fringe} for a proof.
The last unknown variable that we need to compute to apply Lemma \ref{lem:PFmidX} is $q=\P(B_v \mid \kappa=k, \mathcal E_x)$, the probability that a subtree $\mathcal F_v$ of a child $v$ of $\mathrm{root}(\mathcal F)$ is a line graph.
\begin{definition}
\label{def:notation}
For an offspring $v$ of the root of $\mathcal F$, let us denote $v$ by $v_0$, and $v_j$ the first offspring of $v_{j-1}$ for $j\ge 1$. In addition, let us denote by $\tau_{v_j}$ the birth time of $v_j$. Let $\tau_{v_j,2}$ denote the birth-time of the second offspring of individual $v_j$ and let $\tau_2=\min(\{ \tau_{v_j,2} \} )$. 
\end{definition}
We condition on the doomsday clock to ring at time $x$ (this the event $\mathcal E_x$). Since  we assumed that $v=v_0$ is an offspring of a root, and $v$ is alive before time $x$, by Lemma \ref{lem:can_use}, the random variable $\tau_{v_0}$ has density $g_x(y)$ defined in equation \eqref{eq:def_gy}. By definition, the event $B_v$ holds if none of the $v_j$ have two offspring until time $x$, hence, we must find $q=\P(\tau_2>x)$. To describe $\tau_2$, the following definition will be useful.
\begin{definition}
\label{def:exp_inc}
Consider a Poisson point process $\Pi:=\{0=\pi_0, \pi_1, \pi_2, \dots\}$ on $\mathbb R^+$ with intensity $\lambda \in \mathbb{R}^+$ and let $(Y_j)_{j\ge 1}$ be an independent collection of exponential variables, independent of $\Pi$, with $Y_j $ having parameter $j\nu \in \mathbb{R}^+$. Let $\zeta:=\min\{j: Y_{j} \le \pi_{j+1}-\pi_{j}\}$. Then, the exponential random variable with Poisson increasing rate is
\begin{equation}
 H_{\lambda,\nu}= \pi_\zeta +Y_{\zeta}.
 \end{equation}
  \color{black}
 \end{definition}
Due to the memoryless property of exponential variables, we can think of $ H_{\lambda,\nu}$ as a single exponential clock, that starts with initial rate $0$ at time $0$, and every time the governing Poisson point process $\Pi$ has a new point, the rate of the clock increases by $\nu$. The next lemma relates $H_{\lambda, \nu}$ to $\tau_2$:
\begin{lemma}
\label{lem:exp_inc}
Recall that $\tau_{v_0}$ has density $g(y)$ defined in equation \eqref{eq:def_gy}, and let $H_{\rho,\rho+\chi}$ be an exponential random variable with Poisson increasing rate as defined in Definition \ref{def:exp_inc} independent of $\tau_{v_0}$. Then,
\begin{equation}
\label{eq:exp_inc}
\P(\tau_2>x) = \P( H_{\rho,\rho+\chi} + \tau_{v_0} >x).
\end{equation}
\end{lemma}

\begin{figure}[h]
  \includegraphics[width=\textwidth]{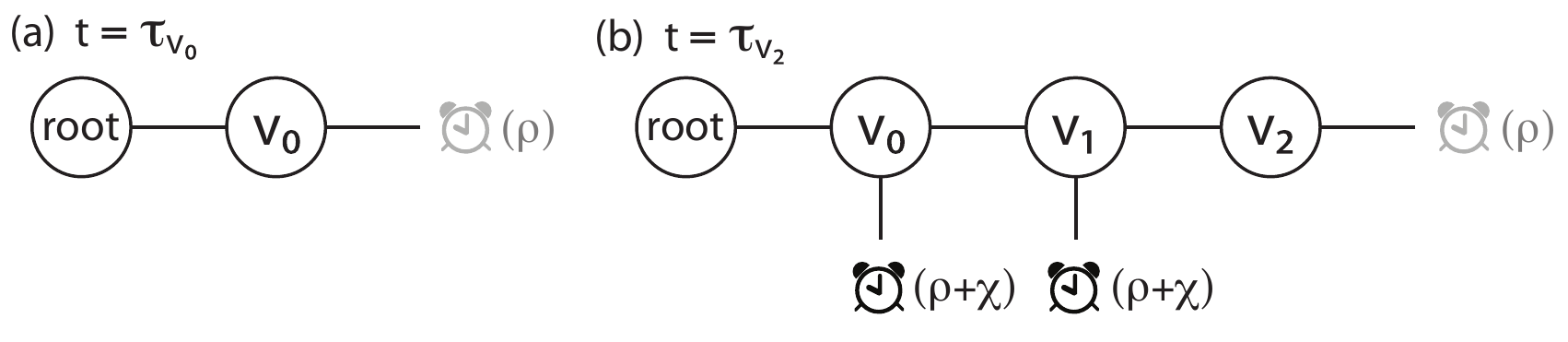}
   \captionsetup{width=\linewidth}
  \caption{Illustration of the proof of Lemma \ref{lem:exp_inc}. Part (a) shows the tree at time $t=\tau_{v_0}$, when only $v_0$ is born, and part (b) shows the tree at time $t=\tau_{v_2}$ assuming $\tau_2>\tau_{v_2}$. If $v_j$ is the last-born vertex at some time $t<\tau_2$, we have an (grey) exponential clock with intensity $\rho$ to govern the Poisson point process $\tau_{v_1},\tau_{v_2}, \dots$, and $j$ (black) exponential clocks with intensity $(\rho+\chi)$ that govern $\tau_2$. If the grey clock rings, a new (black) exponential clock with intensity $(\rho+\chi)$ appears, and $\tau_2$ is the time when the first black clock rings.}
  \label{fig:racing_clocks}
\end{figure}

\begin{proof}
We are going to show that $\tau_2-\tau_{v_0}$ and $H_{\rho,\rho+\chi}$ has the same distribution and both are independent of $\tau_{v_0}$. First, the independence follows from the fact that differences between births of consecutive children in the Crump-Mode-Jagers tree are using independent exponential clocks, see Definition \ref{def:LPARP}.

Next we show that $\tau_2-\tau_{v_0}\ {\buildrel d \over = } \ H_{\rho,\rho+\chi}$.
First we identify the underlying PPP.  In the CMJ tree, by Definition \ref{def:LPARP}, the first offspring of every vertex is governed by an exponential clock with rate $\rho$, hence $(\tau_{v_j}-\tau_{v_0})_{j\ge 0}$ has the same distribution as $(\pi_j)_{j\ge 1}$, a Poisson point process $\Pi$ with intensity $\lambda=\rho$ in Definition \ref{def:exp_inc}. The first offspring form the line-graph emanating from $v=v_0$, see Figure~\ref{fig:racing_clocks}.

The random variable $\tau_2$ is defined as the first time any of the vertices $\{ v_j \mid j\ge 0\}$ have degree at least three. The inequality $\tau_2> \tau_{v_1}$ holds deterministically, because this is the first time any vertex (in this case, $v_0$) can have a second child within the subtree $\mathcal F_{v_0}$.
 This means that until $\tau_{v_1}=\pi_1$, $\tau_2$ cannot happen. Indeed, $\zeta=0$ cannot happen, since the rate of the exponential clock $Y_0$ is $0$, hence $Y_0\le \pi_1-0$ happens with probability $0$.  
 
 By Definition \ref{def:LPARP} again, the rate of arrival of the second child of any individual is $\rho+\chi$.
 For $j\ge 1$, let us look at a scenario when $v_0, v_1, \dots, v_{j-1}, v_j$ are born and forming a line, i.e., they are born, none of them has a second child yet, and $v_{j+1}$ has not been born yet. That is, we look at a time $t\in (\tau_{v_j}, \tau_{v_{j+1}})$.
In this scenario, all of the vertices $v_0, v_1, \dots, v_{j-1}$ are waiting for their  second offspring to be born, hence the total rate of arrival of the second offspring  is governed by an exponential clock with parameter $j(\rho+\chi)$.

As a result, between $\tau_{v_j}$ and $\tau_{v_{j+1}}$ the random variable $\tau_2$ can be described as an $\mathrm{Exp}(j(\rho+\chi))$ random variable (see Figure \ref{fig:racing_clocks}). With $\nu=\rho+\chi$, the random variable $Y_j$ in Definition \ref{def:exp_inc} is also an $\mathrm{Exp}(j(\rho+\chi))$ random variable. 

By the memoryless property of exponential variables, conditioned that $\tau_2>\tau_{v_j}$, $\tau_2$ happens before $\tau_{v_{j+1}}$ if the $\mathrm{Exp}(j(\rho+\chi))$ variable is less than $\tau_{v_{j+1}}-\tau_{v_j}$. Since $\tau_{v_{j+1}}-\tau_{v_j}\ {\buildrel d\over =}\  \pi_{j+1}-\pi_j$, this inequality can be expressed as 
$Y_j\le \pi_{j+1}-\pi_j$. 

In other words, if $Y_j\le  \pi_{j+1}-\pi_j$, then $j$ is the index of the last vertex $v_j$ that is born before $\tau_2$, and $\tau_2=\tau_{v_j} + Y_j$. Otherwise, if $Y_j > \tau_{v_{j+1}}-\tau_{v_j}$, the value of $Y_j$ is irrelevant, $\tau_{v_{j+1}}$ is born before any of the $v_0, v_1, \dots, v_{j-1}$ has a second child, and the rate of getting a second child on the line present goes up by $\rho+\chi$ since now $v_j$ is also waiting for his second offspring to be born. By the memoryless property, we can restart the clocks and use a new exponential variables for comparison. Hence, we move on to the next index $j+1$. The random variable $\zeta$ describes the first index $j$ for which $Y_j\le \tau_{v_{j+1}}-\tau_{v_j}$, which is the first (and only) ``relevant'' index. Then,
$$\tau_2-\tau_{v_0} = Y_\zeta + \tau_{v_\zeta} -\tau_{v_0}\  {\buildrel d \over =}\  Y_\zeta + \pi_\zeta,$$
which is precisely what we needed.
\end{proof}

Although for the proofs we will only need $H_{\rho,\rho+\chi}$, we find the tail distribution of $H_{\lambda,\nu}$ in a general form.

\begin{lemma}
\label{lem:general_form}
The tail distribution of $H_{\lambda,\nu}$ is given by
\begin{equation}
\P( H_{\lambda,\nu}>t)=\exp\Big\{-\lambda t + \frac{\lambda}{\nu} (1-e^{-\nu t})\Big\}.
\end{equation}
\end{lemma}
\begin{proof}[Proof of Lemma \ref{lem:general_form}]
Let us condition on the number of points in the Poisson point process $\pi_1, \pi_2, \dots$ before time $t$, which is just a Poisson random variable with intensity $\lambda t$. We have
\begin{align}
\P( H_{\lambda,\nu}>t)&=\sum\limits_{k=1}^\infty \frac{(\lambda t)^k e^{-\lambda t}}{k!} \E\big[e^{-\nu(\pi_2-\pi_1)}e^{-2\nu(\pi_3-\pi_2)} \dots e^{-(k-1)\nu(\pi_k-\pi_{k-1})}e^{-k\nu(t-\pi_{k})}\big],
\end{align}
where the expectation is over the random points $\pi_1, \dots, \pi_k$. By standard properties of the Poisson point process (in the spirit of Lemma \ref{lem:can_use} with $\rho=\lambda$ and $\chi=0$), we can sample the points $\pi_1, \dots, \pi_k$ by sampling $k$ points uniformly from interval $[0,t]$ and then indexing them such that $\pi_1 < \dots < \pi_k$. Then, by a telescopic cancellation we obtain
\begin{align}
\P( H_{\lambda,\nu}>t)&=\sum\limits_{k=1}^\infty \frac{(\lambda t)^k e^{-\lambda t}}{k!} \E\big[e^{\nu(\pi_1+\dots +\pi_{k})}\big]e^{-k\nu t}.
\end{align}
Since each $\pi_j$ appears exactly once in the sum, and we can forget about their ordering. Then, the $\pi_j$ become independent uniform random variables on $[0,t]$, and we can simplify to
\begin{align}
\P( H_{\lambda,\nu}>t)&=\sum\limits_{k=1}^\infty \frac{(\lambda t)^k e^{-\lambda t}}{k!} \E\big[e^{\nu t U[0,1]}\big]^{ k} e^{-k \nu t}\nonumber \\
&=\sum\limits_{k=1}^\infty \frac{(\lambda t)^k e^{-\lambda t}}{k!}e^{-k\nu t}\left(\int\limits_0^1 e^{\nu tx}\diff x\right)^{k} \nonumber \\
&=\sum\limits_{k=1}^\infty  \frac{(\lambda t)^k e^{-\lambda t}}{k!} e^{-k \nu t}\frac{(e^{\nu t}-1)^{k}}{(\nu t)^{k}}.
\end{align}
Now simply cancelling the appropriate terms and factoring out the term not depending on $k$ we reach the final result
\begin{align}
\P( H_{\lambda,\nu}>t)&=e^{-\lambda t}\sum\limits_{k=1}^\infty \frac{\lambda^k \nu^{-k}}{k!} \left( \frac{e^{\nu t}-1}{e^{\nu t}}\right)^{k} \nonumber \\
&=\exp\!\Big(\!-\lambda t + \frac{\lambda}{\nu} (1-e^{-\nu t})\Big).
\end{align}
\end{proof}
We proceed computing $\P(\mathcal F \in \mathcal P_K)$ in \eqref{eq:BST_p_1}. In order to do this, we make use of Lemma \ref{lem:PFmidX}, that requires the conditional generating function of $\kappa\mid \mathcal E$, that we identified in Lemma \ref{lem:generating} when we take $\mathcal E_x=\{\tau\in (x, x+\diff x)\}$. It remains to calculate $1-q=\P(\tau_2 <x)$ that is needed as the argument of the generating function.
So, we combine Lemmas \ref{lem:exp_inc} and \ref{lem:general_form} to find $q=\P(\tau_2 >x)$. By Lemma \ref{lem:exp_inc}, we must compute the convolution of $H_{\rho, \rho+\chi}$ and the random variable with density $g_x(y)$ defined in equation \eqref{eq:def_gy}, which gives
\begin{align*}
q=\P(\tau_2 >x)&=\frac{1}{Z_{g}(x)}\int\limits_0^x \mathrm{e}^{\chi(x-t)}\exp\Big(-\rho t+\frac{\rho}{\rho+\chi}-\frac{\rho \e^{-(\rho+\chi)t}}{\rho+\chi}\Big) \diff t
\end{align*}
We make the substitution $u=\frac{\rho}{\rho+\chi}\e^{-(\rho+\chi)t}$, which gives $t=-\log(\frac{\rho+\chi}{\rho}u)/(\rho+\chi)$ and $\diff t=\frac{-1}{u(\rho+\chi)}\diff u$, to get
\begin{align}
\label{eq:q_general}
q=\P(\tau_2 >x)&=\frac{\e^{\chi x+\frac{\rho}{\rho+\chi}}}{Z_{g}(x)} \int\limits_{\frac{\rho}{\rho+\chi}}^{ \frac{\rho}{\rho+\chi}\e^{-(\rho+\chi)x}}  \frac{\rho+\chi}{\rho} u\e^{-u} \frac{-1}{u(\rho+\chi)}\diff u \nonumber \\
&=\frac{\e^{\chi x+\frac{\rho}{\rho+\chi}(1-\e^{-(\rho+\chi)x})} - \e^{\chi x}}{\rho Z_{g}(x)},
\end{align}
where $Z_{g}(x)=\int_0^x e^{\chi y} \diff y$ is from Lemma \ref{lem:can_use}.
Finally, we are ready to apply Lemma \ref{lem:PFmidX}. Let us assume $\chi\ne 0$. The $\chi=0$ case will be handled in Section \ref{sub:rrt} below. 
\begin{proof}[Proof of Theorem \ref{thm:LPAT}]
Substituting equations \eqref{eq:gen_fn}, \eqref{eq:P1} into  \eqref{eq:PFmidX} we obtain
\begin{align}
\label{eq:CMJ_PFmidX_sub}
\P(\mathcal{F} \in \mathcal{P}_K \mid \mathcal E_x) &=
1-G_{\kappa \mid \mathcal{E}}(1-q)-q\P(\kappa=1 \mid \mathcal{E}) \nonumber\\
&=1-(\e^{\chi x} + (1-\e^{\chi x})(1-q))^{-\rho/\chi}+q\frac{\rho}{\chi}  (1-\e^{\chi x})\e^{-x(\rho+\chi)} \nonumber \\
&=1-(1-q(1-e^{\chi x})))^{-\rho/\chi} +\frac{\rho}{\chi}q(1-\e^{\chi x})\e^{-x(\rho+\chi)}.  \end{align}
Now, using the value $q$ from \eqref{eq:q_general}, and by $Z_g(x)=(e^{\chi x}-1)/\chi$, we have
\begin{equation}
\label{eq:q(1-e)}
q (1-e^{\chi x})= -\chi\frac{\e^{\chi x+\frac{\rho}{\rho+\chi}(1-\e^{-(\rho+\chi)x})} - \e^{\chi x}}{\rho}.
\end{equation}
Substituting \eqref{eq:q(1-e)} into the second and third terms of equation \eqref{eq:CMJ_PFmidX_sub}, the formula becomes:
\begin{align*}
\P(\mathcal{F} \in \mathcal{P}_K \mid \mathcal E_x)&=1-\left(1+\chi\frac{\e^{\chi x+\frac{\rho}{\rho+\chi}(1-\e^{-(\rho+\chi)x})} - \e^{\chi x}}{\rho} \right)^{-\rho/\chi} -\e^{-\rho x+\frac{\rho}{\rho+\chi}(1-\e^{-(\rho+\chi)x})} + \e^{-\rho x}.
\end{align*}
We apply the law of total probability with respect to the density of the doomsday clock $\tau$ with rate $\rho+\chi$ to compute
\begin{align}
\P(\mathcal{F} \in \mathcal{P}_K ) &= \int\limits_0^\infty (\rho+\chi)\e^{-x(\rho+\chi)} \P(\mathcal{F} \in \mathcal{P}_K \mid \mathcal E_x) \diff x \nonumber \\ 
&= 1-\int\limits_0^\infty (\rho+\chi)e^{-x(\rho+\chi)}\left(1+\chi\frac{e^{\chi x+\frac{\rho}{\rho+\chi}(1-\e^{-(\rho+\chi)x})} - \e^{\chi x}}{\rho} \right)^{-\rho/\chi} \diff x \nonumber \\
&\ \ \  \quad - \int\limits_0^\infty (\rho+\chi)\e^{-x(\rho+\chi)} \e^{-\rho x+\frac{\rho}{\rho+\chi}(1-\e^{-(\rho+\chi)x})} \diff x + \int\limits_0^\infty (\rho+\chi)\e^{-x(\rho+\chi)} \e^{-\rho x} \diff x.
\end{align}
Let us denote the three integrals on the right hand side by $I_1, I_2, I_3$, respectively.
The third integral can be computed explicitly as
\begin{equation}
I_3=\int\limits_0^\infty (\rho+\chi)\e^{-x(\rho+\chi)} \e^{-\rho x} \diff x=\frac{\rho+\chi}{2\rho+\chi},
\end{equation}
and we observe that this term equals $\P(\mathcal F \in \mathcal P_L)$ in \eqref{eq:leaf}, and hence it cancels when substituted back into equation \eqref{eq:BST_p_1}. So, for  \eqref{eq:BST_p_1}, we obtain the result for linear preferential attachment trees
\begin{equation}\label{eq:temp}
\begin{aligned}
\frac{\beta(\mathcal{T}_n^{(\rho, \chi)})}{n} &\xrightarrow{a.s. }-1+\int\limits_0^\infty (\rho+\chi)\e^{-x(\rho+\chi)}\left(1+\chi\frac{\e^{\chi x+\frac{\rho}{\rho+\chi}(1-\e^{-(\rho+\chi)x})} - \e^{\chi x}}{\rho} \right)^{-\rho/\chi} \diff x \\
&\ \ \  \qquad + \int\limits_0^\infty (\rho+\chi)\e^{-x(\rho+\chi)} \e^{-\rho x+\frac{\rho}{\rho+\chi}(1-\e^{-(\rho+\chi)x})} \diff x .
\end{aligned}
\end{equation}
This is the general formula for $(\rho, \chi)$ when $\chi\neq 0$. This also finishes the proof of Theorem \ref{thm:LPAT}, since the formula in \eqref{eq:thm_LPAT} is recovered when $\chi=1$.
\end{proof}
 Now we evaluate this further for the special case $\chi=-1$, and  obtain the metric dimension of $m$-ary increasing trees ($\rho=m \in \mathbb{N}$, and $\chi=-1$).
\begin{proof}[Proof of Theorem \ref{thm:BST}]
When $\chi=-1$ and $\rho=m \in \mathbb{N}$, equation \eqref{eq:temp} simplifies to 
\begin{equation}
\label{eq:temp_simpl}
\begin{aligned}
\frac{\beta(\mathcal{T}_n^{(m,-1)})}{n} &\xrightarrow{a.s. }-1+\int\limits_0^\infty (m-1)e^{-x(m-1)}\left(1-\frac{e^{- x+\frac{m}{m-1}(1-e^{-(m-1)x})} - e^{- x}}{m} \right)^{m} \diff x \\
&\ \ \  \qquad + \int\limits_0^\infty (m-1)e^{-x(m-1)} e^{-m x+\frac{m}{m-1}(1-e^{-(m-1)x})} \diff x .
\end{aligned}
\end{equation}
In the first row, the last bracket is of the form $(1-\mu+\nu)^m$, that we expand using the trinomial formula:
\[ (1-\mu+\nu)^m=\sum_{(i,j)\in \mathbb N^2:i+j\le m} \binom{m}{i,j} (-1)^i\mu^i \nu^j. \]
We apply this formula with $\mu=e^{- x+\frac{m}{m-1}(1-e^{-(m-1)x})}/m$ and $\nu={e}^{-x}/m$. After collecting terms, and taking into account that the integral in the last row of  equation \eqref{eq:temp_simpl} can be merged with the term corresponding to  $(i,j)=(1,m-1)$ of the expansion, changing the coefficient, we arrive at 
\begin{align}
\label{eq:BST_expand}
\frac{\beta(\mathcal{T}_n^{(m,-1)})}{n} &\xrightarrow{a.s. }   \sum_{\substack{0\le i+j \le m }} a_{i,j} \int\limits_0^\infty  e^{-b_{i,j}x + i\frac{ m}{m-1}(1-e^{-(m-1)x})} \mathrm \diff x,
\end{align}
where
\begin{alignat}{2}
a_{1,m-1} &= (m-1)\left(1-\frac{m}{m^{m}} \right)\\
a_{i,j} &= (m-1)\frac{(-1)^{i}}{m^{i+j}}\binom{m}{i,j} \qquad\qquad  \text{if $(i,j)\ne (1,m-1)$} \\
b_{i,j} & = (m-1)+i +j  .
\end{alignat}
For $i=0$, the coefficient $im/(m-1)$ of the doubly-exponential term in equation \eqref{eq:BST_expand} in the exponent is $0$, hence these terms simplify. We sum over the $i=0$ terms in $j$, and perform the integration to obtain
\begin{equation}
\label{eq:BST_i0}
\sum_{j=0}^m \frac{m-1}{m^j}\binom{m}{j} \int\limits_0^\infty e^{-(m-1+j)x} \mathrm \diff x=\sum_{j=0}^m \frac{m-1}{m^j}\binom{m}{j}\frac{1}{m-1+j}.
\end{equation}
Observe that the $j=0$ term is $1$, and hence cancels the $-1$ in the first term of the right hand side of equation \eqref{eq:temp_simpl}.
For the integral indexed by $(i\ne0,j)$ we can substitute $u=i \frac{m}{m-1}e^{-(m-1)x}$ which gives $x=-\log(\frac{m-1}{im}u)/(m-1)$ and $\mathrm \diff x=\frac{-1}{u(m-1)} \diff u$, to obtain
\begin{align}
\label{eq:BST_ine0}
\sum_{\substack{0\le i+j \le m \\ i\ne 0 }} a_{i,j} \int\limits_0^\infty  e^{-b_{i,j}x + \frac{ im}{m-1}(1-e^{-(m-1)x})}  \diff x &=  \sum_{\substack{0\le i+j \le m\\ i\ne 0 }} a_{i,j}' \int\limits_{\frac{im}{m-1}}^0 u^{b_{i,j}'}  e^{-u}  \diff u \nonumber \\
&= \sum_{\substack{0\le i+j \le m\\ i\ne 0 }} -a_{i,j}' \gamma \left(b_{i,j}'+1, \frac{im}{m-1} \right),
\end{align}
where
\begin{alignat}{2}
a_{i,j}' &= a_{i,j} e^{\frac{im}{m-1}} \left(\frac{m-1}{im}\right)^{\frac{b_{i,j}}{m-1}}\frac{-1}{m-1}  \label{eq:BST_a'}  \\
b_{i,j}' & = \frac{b_{i,j}}{m-1}-1= \frac{i+j}{m-1}. \label{eq:BST_b'}
\end{alignat}
Combining equations \eqref{eq:BST_expand}-\eqref{eq:BST_b'} gives the formula
\begin{equation}
\label{eq:mary_final}
\frac{\beta(\mathcal{T}_n^{(m,-1)})}{n} \xrightarrow{a.s. } \sum_{j=1}^m \frac{m-1}{(m-1+j)m^j}\binom{m}{j} - \sum_{\substack{0\le i+j \le m\\ i\ne 0 }} a'_{i,j} \gamma \left(\frac{i+j}{m-1}+1, \frac{im}{m-1} \right),
\end{equation}
which agrees with equation \eqref{eq:BST} in Theorem \ref{thm:BST} with $A_{i,j}=-a'_{i,j}$.

For the binary search tree, that is, $m=2$ we evaluate the coefficients in equations \eqref{eq:BST_i0} and \eqref{eq:BST_a'}  numerically. Starting with equation \eqref{eq:BST_i0}, then proceeding to the coefficients $a_{1,1}', a_{1,0}', a_{2,0}'$, we get
\begin{align*}
\sum_{j=1}^m \frac{m-1}{m^j}\binom{m}{j}\frac{1}{m-1+j}&=\frac12\cdot 2 \cdot \frac12+\frac13 \cdot \frac14=\frac{7}{12} \\
-a_{1,1}'&=\frac12 e^2 \frac{1}{2^3}=\frac{e^2}{2^4}\\
-a_{1,0}'&=-\frac12 2 e^2 \frac{1}{2^2}=-\frac{e^2}{2^2}\\
-a_{2,0}'&=\frac14 e^4 \frac{1}{4^2}=\frac{e^2}{2^8}.
\end{align*}
Substituting these values into equation \eqref{eq:mary_final} gives
\begin{equation}
\frac{\beta(\mathcal{T}_n^{(2,-1)})}{n} \xrightarrow{a.s. } \frac{\e^2}{2^4}\gamma(3,2)- \frac{\e^2}{4}\gamma(2,2)+\frac{e^4}{2^8}\gamma(3,4)+\frac{7}{12}=\frac{233 - 48 \e^2 + 3 \e^4}{384}.
\end{equation}
\end{proof}
Next we proceed with the random recursive tree ($\chi=0$ and $\rho=1$). The proof is analogous to the proof of Theorem \ref{thm:LPAT}. We proceed from formula \eqref{eq:q_general}.
\begin{proof}[Proof of Theorem \ref{thm:RRT}]
\label{sub:rrt}
In this case, equation \eqref{eq:q_general} yields
\begin{equation}
\label{eq:q_rrt}
q=\frac{\e^{\chi x+\frac{\rho}{\rho+\chi}(1-\e^{-(\rho+\chi)x})} - \e^{\chi x}}{\rho Z_{g}(x)}=\frac{\e^{(1-\e^{-x})} - 1}{x}.
\end{equation}
Substituting equations \eqref{eq:gen_fn}, \eqref{eq:P1} for $\chi=0$ and \eqref{eq:q_rrt} into equation \eqref{eq:PFmidX} now gives
\begin{align}
\label{eq:RRT_PFmidX_sub}
\P(\mathcal{F} \in \mathcal{P}_K \mid \mathcal E_x) &= 1-\e^{-x(1-(1-q))}-qx\e^{-x} =1-\mathrm{exp}(1-\e^{1-\e^{-x}}) -\e^{1-x-\e^{-x}}+\e^{-x}.
\end{align}
In this case, $\tau$ is exponential with rate $\rho=1$. We apply the law of total probability to compute
\begin{align}
\P(\mathcal{F} \in \mathcal{P}_K ) &= \int\limits_0^\infty  \P(\mathcal{F} \in \mathcal{P}_K \mid \mathcal E_x) \e^{- x} \diff x= 1- \int\limits_0^\infty e^{-x} \big(1-\mathrm{exp}(e^{1-e^{-x}}) +e^{1-x-e^{-x}}-e^{-x}\big) \diff x.
\end{align}
We make the substitution $u=e^{-x}$, which gives $x=-\log(u)$ and $\diff x=-\frac{1}{u}\diff u$, to get
\begin{align}
\P(\mathcal{F} \in \mathcal{P}_K ) &= 1+ \int\limits_{1}^{0} (e^{1-e^{1-u}} +ue^{1-u}-u) \diff u = 1-\int\limits_{0}^{1} e^{1-e^{1-u}} \diff u - e\gamma(2,1)+\frac12,
\end{align}
where $\gamma$ was defined in equation \eqref{eq:gammadef}. Furthermore, we substitute $v=e^{1-u}$ in the integral still remaining, which gives $u=1-\log(v)$ and $\diff u=-\frac{1}{v}\diff v$, to get
\begin{equation}
\int\limits_{0}^{1} e^{1-e^{1-u}} \diff u = -\int\limits_{e}^{1} v^{-1}e^{1-v} \diff v =e\int\limits_1^e v^{-1}e^{-v} \diff v.
\end{equation}
Substituting back into equation \eqref{eq:BST_p_1} we obtain the final result
\begin{align}
\frac{\beta(\mathcal{T}_n^{(1,0)})}{n} &\xrightarrow{a.s. } \frac12-\left(1-e\int\limits_1^e v^{-1}e^{-v} \diff v - e\gamma(2,1)+\frac12\right)=e\left(\int\limits_1^e v^{-1}e^{-v} \diff v+\gamma(2,1)\right)-1.
\end{align}
This finishes the proof.\end{proof}
         


\bibliographystyle{abbrv}
\bibliography{literature}

\end{document}